\documentclass[a4paper,12pt]{article}
\usepackage[utf8]{inputenc}
\usepackage{graphicx} 
\usepackage{tikz}
    \usetikzlibrary{positioning}
    \usetikzlibrary{cd}
    \tikzset{main node/.style={circle,fill=blue!20,draw,minimum size=1cm,inner sep=0pt},}

\usepackage[all]{xy}
\usepackage{subcaption}
\usepackage{amsmath}
\usepackage{amsthm}
\usepackage{thmtools}
\usepackage{amssymb}
\usepackage{amsfonts}
\usepackage{hyperref}
\usepackage{cleveref}
\usepackage{comment}
\usepackage{fullpage}

\newtheorem{defi}{{\it \bf{Definition.}}}[section]
\newtheorem{prop}{{\it \bf{Proposition.}}}[section]

\newtheorem{cor}{{\it \bf{Corollary.}}}[section]
\newtheorem{lem}{{\it \bf{Lemma.}}}[section]
\newtheorem{teo}{{\it \bf{Theorem.}}}[section]

\newtheorem{rem}{{\it \bf{Remark.}}}[section]

\newcommand{\rarr}{\rightarrow}

\newcommand{\bbZ}{\mathbb{Z}}
\newcommand{\bbQ}{\mathbb{Q}}
\newcommand{\bbR}{\mathbb{R}}

\newcommand{\mathset}[1]{{\left\{#1\right\}}}
\newcommand{\absolute}[1]{\left\lvert#1\right\rvert}

\newcommand{\ToDo}[1]{{{\bigskip\noindent\bf To Do: }#1\medskip}}
\newcommand{\ToCite}[1]{{{\bigskip\noindent\bf To Cite: }#1\medskip}}

\newif\ifworkinprogress
    \workinprogresstrue

\title{$p$-Adic Diffusion and Random Walks on $[0,1]$}

\author{Patrick Erik Bradley\\
 Paulina Halwas
 \\Ell Ari Nitsche
 \\\'Angel Mor\'an Ledezma
\\[2mm]
Geodetic Institute Karlsruhe
\\
and
\\
Institute of Photogrammetry and Remote Sensing
\\
Karlsruhe Institute of Technology
\\
Englerstr.\ 7
\\
76131 Karlsruhe, Germany
}
\date{\today}

\begin{document}

\maketitle

\begin{abstract}
Integral operators on the real unit interval are constructed as  transported from  ones on the $p$-adic unit disc via the Monna map. This gives rise to  strong Markov processes on $[0,1]$, whose paths
are right-continuous and have no discontinuities other than jumps.
The spectra of the corresponding diffusion operators, whose kernel functions depend on a $p$-adic distance, are calculated. Further, the solution to the Cauchy problem for their heat equations are approximated via continuous-time random walks on finite sets coming from a hierarchical partition of the unit interval induced by the $p$-adic distance. The transport of  $p$-adic diffusion to the real domain via the Monna map gives rise to a simple visualisation method. Illustrations of concrete examples are undertaken in the end. 
\end{abstract}

\tableofcontents

\setlength{\parindent}{0em}
\setlength{\parskip}{1em}

\section{Introduction}

The emerging field of $p$-adic analysis and its application to stochastic processes
is quite recent, as it is ``just'' a few decades old, and at the same time, it ``feels'' like it is totally disconnected from real analysis   and its own application to stochastic processes.
However, there is a link between these two seemingly detached areas of mathematics: not only do these two fields of pure mathematics have a lot of common ground in operator theory and functional analysis, but also are they joined by a direct connection between the $p$-adic numbers $\mathbb{Q}_p$ and the real numbers $\mathbb{R}$ through a map which has only sporadically  gained interest so far: the \emph{Monna map}, introduced in \cite{Monna1952}. An early application in the statistics of numbers is \cite{Meijer1967}, a more recent one is \cite{Weiss2025}. This surprising  fact is now in a process of being updated: during the process of writing down the results of this research, the publication \cite{Wilson2026} appeared which for the first time acknowledges the implications of the Monna map in establishing an isometric isomorphism between the Hilbert spaces $L^2(\mathbb{Q}_p,\mu)$ and $L^2(\mathbb{R},\lambda)$, where $\mu$ is the Haar measure and $\lambda$ the Lebesgue measure, both normalised to unity on the unit disc, and the unit interval, respectively. In \cite{Wilson2026}, this isometry plays out on the level of real-valued functions on the unit interval in order to obtain a non-Archimedean description of Deep Neural Networks via functions on the $p$-adic unit disc. However, to us it seems (although not explicitly formulated in that way) meet and just to acknowledge both the original work of Monna in \cite{Monna1952}, as well as Kozyrev in \cite{Kozyrev2004}, in envisioning this natural correspondence between the $p$-adic and the real number worlds as being much closer to us than merely on the horizon. In particular, Kozyrev's observation, that his novel $p$-adic wavelets are transported via the Monna map to generalised forms of the Haar wavelet, deserves mentioning, as it forms an important part of the ground upon which this article is built upon. 
\newline

In this article, we introduce a random process on the probability space $([0,1],\lambda)$ where $\lambda$ denotes the usual Lebesgue measure with total mass equaling to one. This process can be understood as a limiting process of a sequence of continuous-time discrete-space processes. All these processes are induced via carrying over a $p$-adic counterpart via the Monna map. This map thus acts as a ``translator'' between the $p$-adic theory and the real-domain theory of stochastic processes. In particular, since the kernel functions used are all radial w.r.t.\ the $p$-adic absolute norm, the Kozyrev wavelet eigenvalue formula can be used to obtain a fast method for computing eigenvalues of Laplacian integral operators of these kind. This is exploited in the end in two example processes which are approximate $p$-adic diffusion processes on the real unit interval.
\newline

In order to see existing research in mathematical physics related to this article, let us mention that
in stochastic thermodynamics the dynamics of mesoscopic physical systems subject to random interactions with a heat reservoir is studied, cf.\ e.g.\ 
\cite{PP2021,FE2025}. In this context, as in  material kinetics \cite{Mauro2021},  such random models are usually described by a probability function $p_I(t)$ that describes the probability of finding the system in a discrete state $I$ at a given time $t$. The evolution of $p_I(t)$ over time is usually described by a continuous time Markov chain. 
This process, viewed here as jumping between meta-basins on an energy landscape \cite{MS2012}, is fully charaterized by its master equation, cf.\ \cite[Ch.\ V]{Kampen2007}. Let $A$ be a finite set of states. The so called \textit{master equation} 
\begin{equation}
\label{eqn:1}
    \frac{d }{dt}p_I(t)=\sum_{J\in A} \{w_{I,J}p_J(t)-w_{J,I}p_I(t)\}
\end{equation}
describes the time evolution of the function $p_I(t)$, the probability of finding the system at state $I\in A$ at time $t$. The numbers $w_{I,J}\geq 0$ denotes the probability transition rate per unit of time from state $J$ to the state $I$. These numbers are also called jump rates. It is clear that a master equation is determined by the matrix of jump rates 
\[
W=[w_{I,J}]_{I,J}\in\mathbb{R}_+^{\absolute{A}\times\absolute{A}}\,.
\]
Moreover, it is useful to represent a master equation via a jump network or jump graph, where the nodes represent the states $I$, the arrows represent the possible jumps, and the weights  correspond to jump rates \cite[Chapter 2]{PP2021}.  Solving the master equation explicitly can be a highly non trivial problem and in practice is often inefficient. Therefore, usual methods involve random trajectories for example using the Monte Carlo method or the Gillespie algorithm \cite{Gillespie1977}.
\newline

It turns out that if the jump rates depend on a $p$-adic distance between the states, then the explicit form of the eigenvalues given in \cite{Kozyrev2004} can be used to effect a linear complexity, both in space and time.
Under certain regularity conditions, the time-complexity for general ultrametric kernel functions can with good reason expected to also be linear, as there is the wavelet eigenvalue formula also in this case, cf.\ \cite{XK2005}.
Descriptions of such processes via $p$-adic analysis can be found in the mathematised sciences \cite{Kozyrev2011},
and in mathematics itself, where in particular the heat equation on $p$-adic integers as an example of diffusion on a compact space is studied, cf.\ \cite{Kochubei2018,PW2025}. Based on Z\'u\~{n}iga's approach to Turing patterns on graphs via networks connecting finitely many disjoint $p$-adic discs, cf.\ \cite{ZunigaNetworks,Zuniga2022}, the recent research of two authors of this article builds upon those ideas in order to study local ultrametric approximations of graph
diffusion, \cite{locUltraGraphs},
time-changing graphs and applications
\cite{nonAutonomous,Ledezma-Energy}, hearing the shape of graphs
 \cite{BL_shapes_p}, diffusion on multi-topologies
\cite{UltrametricTopoIndex}, in order to name some. The dissertation \cite{AngelDiss} develops in depth spectral and stochastic methods  on ultrametric spaces and applications 
in the sciences.
The very recent article \cite{Wilson2026} uses Monna's result on the Haar-to-Lebesgue measure correspondence between the $p$-adic integers to the real unit interval in order to develop a $p$-adic formulation of Deep Neural Networks.
This already shows the potential of the Monna map in transporting information back and forth between the $p$-adic and the real number domains, and this is further exploited here to carry over whole Markov processes, and thus to obtain a simple visualisation method for $p$-adic diffusion processes. 
\newline

Studies of random walks on ultrametric Cantor sets in $\mathbb{R}$ often naturally lead to investigating the relationship between ultrametric trees and ultrametric Brownian motion, cf.\ \cite{Kigami2010,Kigami2013,KT2024}. These are so far formulated in $p$-adic terms alone, whereas here, it is via an integral operator on a domain within the real numbers with kernel function taking its input via the Monna map. In this way, obtain random walks having as state space the whole real unit interval, instead of a Cantor set within. And it is its $p$-regular hierarchical property which allows  natural finite approximations via finite random walks on the boundaries of $p$-regular trees, also in the the Cauchy problems and their solutions. This is somewhat different from  usual Brownian motion on the real unit interval.
\newline

Our main result is stated via the following integral operator
given by a kernel function
$k(x,y)$ defined via the Monna map $\rho$ as
\[
k(x,y)=f\left(\absolute{\rho^{-1}(x)-\rho^{-1}(y)}_p\right)\,,
\]
and the corresponding integral operator given by
\[
D_fu(x)=\int_0^1k(x,y)(u(y)-u(x))\,dy
=\int_0^1f\!\left(\absolute{\rho^{-1}(x)-\rho^{-1}(y)}_p\right)(u(y)-u(x))\,dy
\]
for functions $u\colon [0,1]\to\mathbb{C}$. It naturally relates to the operator

\begin{equation} \label{eq:padicoperator1}
P_fh(\xi)=\int_{\mathbb{Z}_p}f\!\left(\absolute{\xi-\eta}_p\right)(h(\eta)-h(\xi))\,d\mu(\eta)\,,
\end{equation}
which is an infinite-dimensional version of
(\ref{eqn:1}) on the $p$-adic unit disc $\mathbb{Z}_P$, whose kernel function depends on the $p$-adic norm $\absolute{\cdot}_p$, in the  following way:

{\bf Theorem. 3.2.}
{\em
    Let \(X_t\in [0,1]\) be the strong Markov process attached to the infinitesimal generator \(D_f\) with probability transition \(p_t(x,\cdot)\). Let \(P_f\) be the operator defined on \(L^2(\mathbb{Z}_p,\mu)\) defined by equation (\ref{eq:padicoperator1}). Let \(X^{*}_t:=\rho^{-1}(X_t)\) be the pull-back of \(X_t\). 
    Then the following holds
    \begin{enumerate}
        \item The attached semigroups are unitarily equivalent:
        \begin{equation*}
            e^{tD_f}=\rho^{*-1} e^{tP_f}\rho^{*}
        \end{equation*}
        \item The pull-back \(X^{*}_t\in \mathbb{Z}_p\) is a strong Markov process on \(\mathbb{Z}_p\) with probability transition \(p_t^{*}(y,\cdot)\) satisfying
        \[p_t^{*}(x,B)=p_t(\rho(x),\rho(B)),\]for any Borel set \(B\subset \mathbb{Z}_p\).
    \end{enumerate}
}

This result is surrounded in this article by solving the Cauchy problem to the associated heat equation (Theorem 2.1), a spectral decomposition of $L^2([0,1],\lambda)$ for such an operator coming from the $p$-adics, and including an eigenvalue formula (borrowed from Kozyrev), finite approximation of the solution to the Cauchy problem via solutions to approximate Cauchy problems in the supremum norm (Theorem 3.3), and exemplified approximate $p$-adic diffusion on the real unit interval visualised through the Monna map entering the algorithm in Section 4.
\newline

The following Section 2 begins with preliminaries, briefly introducing the $p$-adic numbers, the Monna map and other concepts needed for the remainder of this article. It contains our first result (Theorem 2.1). Section 3 introduces the $p$-adic integral Laplacian operator on the real unit interval via the Monna map and is devoted to prove Theorems 3.1 and 3.2. Section 4 concludes with an algorithm for obtaining visualisations of these $p$-adic processes, exemplified in a Gaussian  and in a power  kernel function, in both cases using  $p$-adic distance. 

\section{Preliminaries}

The set of rational numbers $\bbQ$ is usually equipped with the absolute value
\(
    |\cdot|: \bbQ \rightarrow \bbR^+_0
\), and  the completion of $\bbQ$ with respect to $|\cdot|$
is the field of real numbers $\bbR$.

There are several other norms one could have chosen.
On such choice is given by any prime number $p$ as explained in what follows.
Every non-vanishing fraction $\frac{a}{b}\in \bbQ \setminus \{0\}$ can now be rewritten
with a suitable $\nu\in\bbZ$
and coprime integers $a'$ and $b'$ as
\[
    \frac{a}{b} = p^\nu \cdot \frac{a'}{b'}.
\]
%
This representation is unique for every rational number,
where $\nu$ is called the \emph{$p$-adic valuation} of $\frac{a}{b}$.
In this way we obtain a function
\(
    \nu: \bbQ\setminus\{0\} \rarr \bbZ
\),
used to defined the $p$-adic norm on $\bbQ$ by setting
\begin{equation}\label{p-adicNorm}
    |x|_p := p^{-\nu(x)}
\end{equation}
for all $x\in\bbQ\setminus\{0\}$ and by necessity $|0|_p := 0$.
The completion of $\bbQ$ via $|\cdot|_p$
is the field of \emph{$p$-adic numbers}
and denoted by $\bbQ_p$.
The $p$-adic norm $\absolute{\cdot}_p$ satisfies the usual properties of a norm on a field, whereby the triangle inequality holds true in a stricter form:
\[
\absolute{x+y}_p\le\max\mathset{\absolute{x}_p,\absolute{y}_p}\,,
\]
which turns $\absolute{\cdot}_p$ into a \emph{Non-Archimedean} norm on $\mathbb{Q}$ and on $\mathbb{Q}_p$, respectively. The corresponding distance
\[
d_p(x,y)=\absolute{x-y}_p
\]
is consequently an ultrametric on $\mathbb{Q}_p$. In particular, $p$-adic discs never strictly overlap, i.e.\ any two of them are either disjoint, or one contains the other.

\begin{rem}
    In fact,
    every non-trivial norm on $\bbQ$
    is equivalent either to the usual Euclidean norm $|\cdot|$
    or to $|\cdot|_p$ for some prime number $p$.
 This is Ostrowski's Theorem \cite{Ostrowski1916}.     Hence any completion of the field $\bbQ$ of rational numbers 
    is  isomorphic either to $\bbR$ or to some $\bbQ_p$ (as topological fields).
\end{rem}

For a more concrete view on $\bbQ_p$, look at the formal Laurent expansion 
\begin{align}\label{p-adicExpansion}
x=\sum_{k=\nu}^\infty x_k p^k.
\end{align}
for some $\nu \in \bbZ$ and digits
$x_k\in\{0,\dots,p-1\}$,
where $x_\nu\neq 0$.

It is not hard to see that a norm like  (\ref{p-adicNorm}) can be defined now for $\nu(x)$ being the index of the first
non-vanishing digit $x_k$ in this formal Laurent series.
It turns out that the set of all such Laurent series equipped with this norm is
 $\bbQ_p$, and 
 one can embed $(\bbQ,|\cdot|_p) \rarr (\bbQ_p,|\cdot|_p)$ 
via an isometry whose image is dense in $\bbQ_p$. In this way,  $\bbQ_p$ is given concretely by Laurent expansions of the $p$-adic numbers. Cf.\ \cite{Gouvea2020} for more on $p$-adic numbers.

\begin{rem}\label{nonUniqueExpansion}
Notice that
\[
\absolute{p^k}_p=p^{-k}
\]
for $k\in\mathbb{Z}$, and hence each Laurent series (\ref{p-adicExpansion}) is convergent in $\mathbb{Q}_p$.
This certainly cannot be said about Laurent series  in $\bbR$ with respect to the standard Euclidean norm. The convergence of each Laurent series in $\mathbb{Q}_p$ is induced by the ultrametricity of the $p$-adic absolute value $| \cdot |_p$.
    Moreover, the digit representation $(\dots,x_0,\dots,x_\nu)$ of a $p$-adic number $x\in\mathbb{Q}_p$ is unique, whereas uniqueness of digit-representation does not hold for the $p$-adic expansions of real numbers as
    \[
        \pm \sum_{k=-\infty}^m x_k p^k\,,
    \]
    because e.g.\ the number $1\in\mathbb{R}$ has two such $p$-adic expansions.
\end{rem}

A bridge between $\mathbb{Q}_p$ and the non-negative real numbers $\mathbb{R}^+$ is given by the so-called \emph{Monna map}:  

\begin{defi}
  The \emph{Monna map}  $\rho:\mathbb{Q}_p \rightarrow \mathbb{R}_+$ is defined by 
$$\rho: \sum_{i=\gamma}^{\infty}x_ip^{i}\mapsto \sum_{i=\gamma}^{\infty} x_i p^{-i-1},$$
where $x_i=0,...,p-1$ and $\gamma\in \mathbb{Z}$.    
\end{defi}

We present some results related to this connection.  The first one is that the Monna map is Lipshitz in the following sense:

\begin{lem}
The map $\rho$ satisfies the inequality
\[
|\rho(x)-\rho(y)|\leq |x-y|_p
\]
for every $x,y\in \mathbb{Q}_p$.  In other words, the Monna map is a $1$-Lipshitz function.
\end{lem}

\begin{proof}
\cite[Lemma 3]{Kozyrev2011}.
\end{proof}

\begin{rem}
In the original paper  \cite{Monna1952} by A.F.\ Monna, one can find a proof that $\rho$ is continuous.
\end{rem}

\begin{lem}\label{PropertiesMonna}
The map $\rho$ satisfies 
\begin{align*}
\rho\left( p^mn+p^k\mathbb{Z}_p\right)&= p^{-m}\rho(n)+[0,p^{-k}]
\\
\rho\left(\mathbb{Q}_p\setminus \{p^mn+p^k\mathbb{Z}_p\}\right)
&=  \mathbb{R}_+ \setminus\left(p^{-m}\rho(n)+[0,p^{-k}]\right)
\end{align*}
for $n\in \mathbb{Q}_p / \mathbb{Z}_p$ and $m,k \in \mathbb{Z}$ with $k\geq m$, and it is a bijection
\[
\mathbb{Q}_p\setminus N_p\cong\mathbb{R}_+\,,
\]
where 
\[
N_p=\mathset{\frac{a}{p^n}\mid a\in\mathbb{N}\setminus\mathset{0},\;n\ge1}\,,
\]
viewed as a subset of $\mathbb{Q}_p$.
\end{lem}

\begin{proof}
The proof is quite straightforward, and more or less explicitly contained in \cite{Monna1952}. A calculation of the first assertion is given in \cite[Lemma 4]{Kozyrev2002}.
\end{proof}

\begin{rem}
The subset $N_p\subset\mathbb{Q}_p$ corresponds to the locus in $\mathbb{R}_+$ where the real $p$-adic expansion is non-unique. 
From Lemma \ref{PropertiesMonna}, it follows that there is an induced  map
\[
\overline{\rho}\colon\mathbb{Q}_p / \mathbb{Z}_p \rightarrow \mathbb{N}
\]
 by the Monna map, and this is
 a bijection.  
\end{rem}

The next results concern the transport of measure and integration via the Monna map.

\begin{lem}\label{Haar2Lebesgue}
The Monna map $\rho$ takes the Haar measure $\mu$ on $\mathbb{Q}_p$ to the Lebesgue measure $\lambda$ on $\mathbb{R}_+$. In other words,   $\rho$ takes measurable sets in $\mathbb{Q}_p$ to measurable sets in $\mathbb{R}_+$, and it holds true that
\[
\int_A d\mu= \int_{\rho(A)} d\lambda
\]
for any measurable set $A\subset \mathbb{Q}_p$.
\end{lem}

\begin{proof}
\cite[Page 8]{Monna1952}. Lemma \ref{PropertiesMonna} exhibits the images of discs as being closed intervals (also discs!), and discs generate the $\sigma$-algebras both, on the $p$-adic, as well as on the real side.
\end{proof}

\begin{rem}\label{isometry}
An immediate consequence of Lemma \ref{Haar2Lebesgue} is that that the pull-back map
\[
\rho^{*}:L^{2}(\mathbb{R}_+,\lambda)\rightarrow L^{2}(\mathbb{Q}_p,\mu),\; 
 \rho^*(f)(x)=f(\rho(x))
 \]
 defines an  isometry between Hilbert spaces. 
\end{rem}
    %


\begin{figure}[ht]
    \centering
\includegraphics[width=1\textwidth]{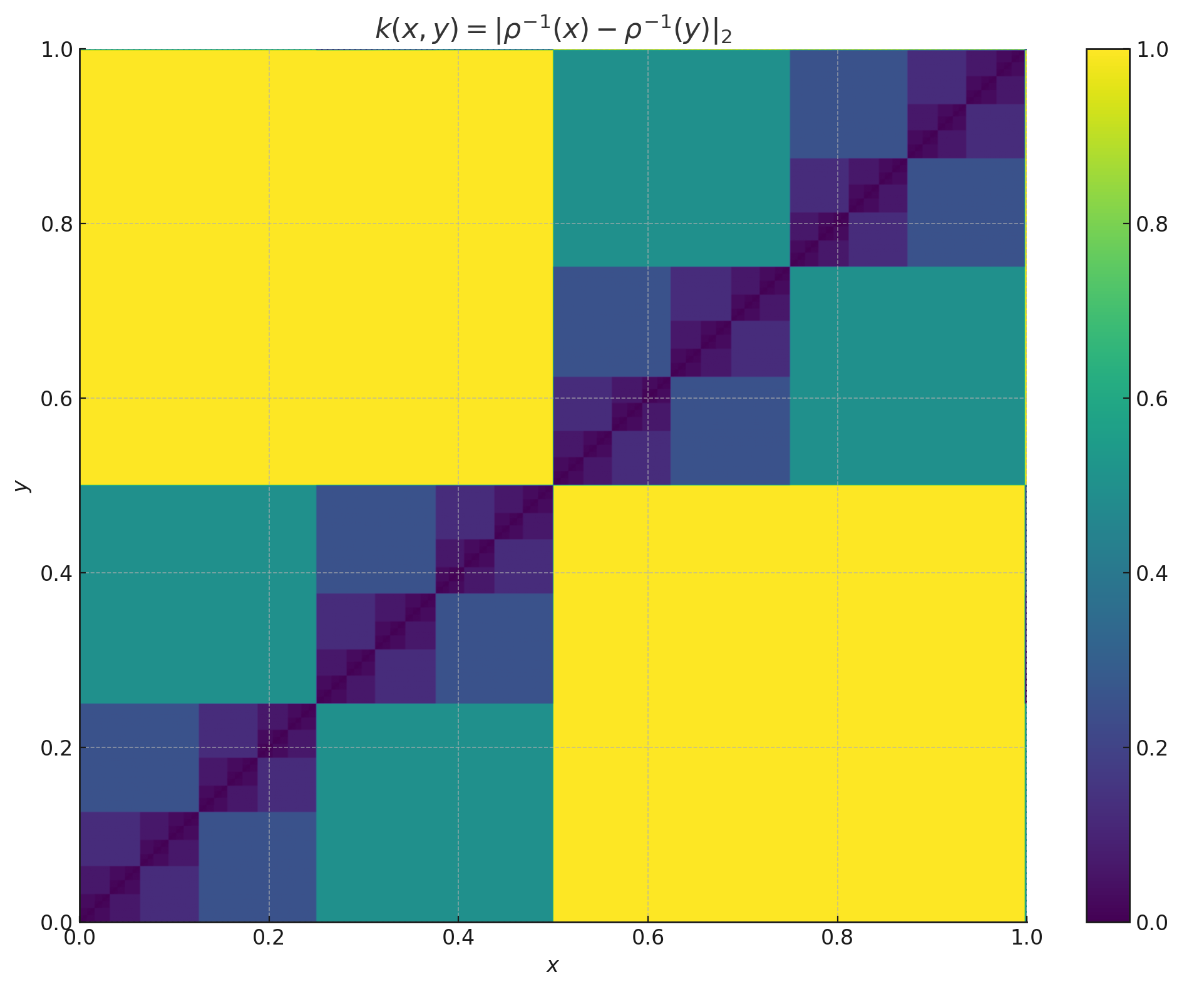}
    \caption{The map $k(x,y)=|\rho^{-1}(x)-\rho^{-1}(y)|_p$ for $p=2$.}
    \label{fig:mesh1}
\end{figure}

\begin{figure}
\includegraphics[width=1\textwidth]{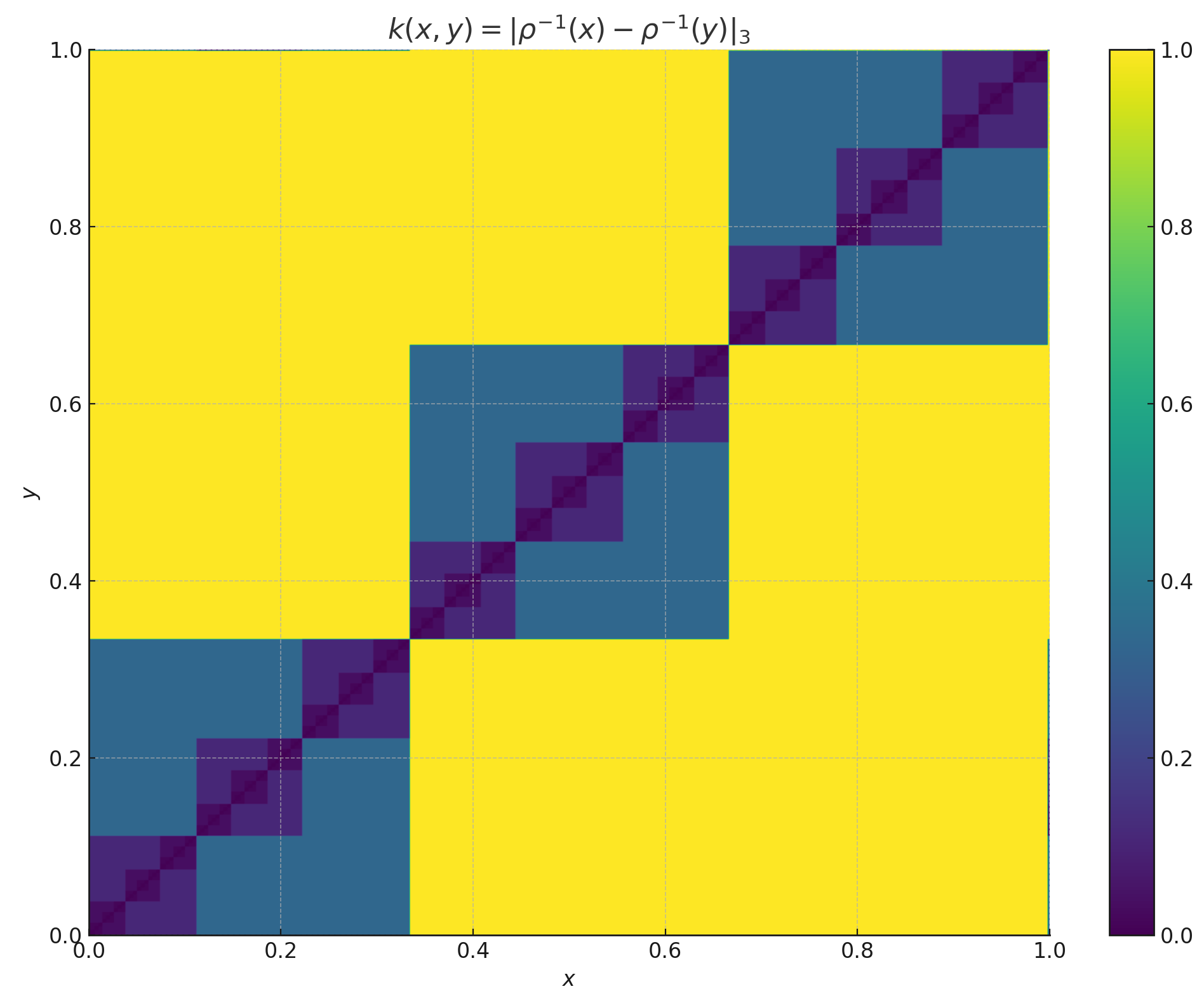}
    \caption{The map $k(x,y)=|\rho^{-1}(x)-\rho^{-1}(y)|_p$ for $p=3$.}
    \label{fig:mesh2}
\end{figure}

\begin{rem}
Lemma \ref{Haar2Lebesgue} is used by S.V.\ Kozyrev to show that real Haar wavelets correspond uniquely to the functions known in $p$-adic analysis as Kozyrev wavelets in the case $p=2$ \cite[Theorem 7]{Kozyrev2002}. The significance of his result is that Kozyrev wavelets form an orthonormal basis of $L^2(\mathbb{Q}_p,\mu)$ consisting of eigenfunctions for the Vladimir operator. And for the  Laplacian integral operators in  the present article, this fact will be used for being able to visualise $p$-adic diffusion via the Monna map, and moreover to enable $p$-adic processes to be viewed as certain kinds of processes on a real domain.   Figures \ref{fig:mesh1} and \ref{fig:mesh2} are  visualisations of the  $p$-adic kernel function transported to the reals:
\[
k(x,y)=\absolute{\rho^{-1}(x)-\rho^{-1}(y)}_p
\]
via the Monna map $\rho$.
\end{rem}

A (generalised) wavelet basis of $L^2(\mathbb{R}_+,\lambda)$ consists of functions of the form 
\begin{equation}\label{genwavelet}
    \psi_{rnj}^{(p)}(x)=p^{-r/2} \Psi_j^{p}(p^{-r}x-n)
\end{equation}
with $r\in \mathbb{Z}$, $n\in \mathbb{N}$, and 
\[
\Psi_j^{p}(x)=\sum_{\ell=0}^{p-1}e^{2\pi\sqrt{-1}  j \ell p^{-1}}1_{[\ell p^{-1},(\ell+1)p^{-1}]}(x)
\]
for $j=1,...,p-1$, and where $1_{A}$ is the indicator function of the set $A\subset \mathbb{R}_+$\,.  

\begin{lem}
The generalised wavelets form an orthonormal basis of $L^2(\mathbb{R}_+,\lambda)$.
\end{lem}

\begin{proof}
The proof for the case $p=2$ generalises in a straightforward manner.
\end{proof}

\begin{rem}
Needless to say,
but it is immediate that the generalised wavelets
pull back to an orthonormal basis of the
Vladimirov operator under the Monna map,
in order to state Kozyrev's result of
\cite[Theorem 7]{Kozyrev2002} for $p > 2$.
\end{rem}

Assume a  function 
\[
w\colon X\times X\to\mathbb{R},\;(x,y)\mapsto w(x,y)
\]
defined on the space $X = [0,1] \subset \mathbb{R}$,
  well-defined and essentially bounded, implying that $w \in L^{\infty}([0,1] \times [0,1])$. Additionally, assume that $w$ is symmetric, i.e.\  $w(x,y) = w(y,x)$. 

\begin{lem}\label{boundedOp}
The assignment     \begin{equation}\label{OurOperator}
        C[0,1]\ni u(x) \mapsto
            \mathcal{W} u(x) = \int_0^1 w(x,y)( u(y) - u(x))\mathop{dy}
    \end{equation}
    is a well defined linear bounded operator on $C[0,1]$.
    \end{lem}

\begin{proof}
 Let    $u\in C[0,1]$. Then due to the boundedness of $w$, it holds true that
 \begin{align*}
 |w(x,y) (u(y) - u(x))| 
        & \leq \max_{x,y\in X} | w(x,y)| \cdot 
        \left(\max_{y \in X} |u(y)| + \max_{x \in X}|u(x)|\right)
\\
 &        \leq \max_{x,y\in X} | w(x,y)| 
        \cdot 2 \cdot\max_{y \in X} |u(y)| 
             = c \|u\|_{\infty}\,,\quad c \in \mathbb{R}\,.
 \end{align*}
This shows that the operator $\mathcal{W}$ is bounded for all $u\in C([0,1])$. 
\end{proof}

\begin{rem}
Notice that $\mathcal{W}$ is a closed operator on $C[0,1]$, because it is a bounded linear operator by Lemma \ref{boundedOp}. Moreover, it is well-defined on $L^{\infty}[0,1]$.
\end{rem}
    
The Cauchy problem for $u(\cdot,t)\in C[0,1]$
    and $u(x,\cdot)\in C^1[0,\infty)$ is spelled out as follows:
\begin{equation}\label{CauchyProblem}
    \begin{cases}
      \frac{\partial u}{\partial t} (x,t)=\mathcal{W}u(x,t)\\
      u(x,0)=u_0(x)\in C[0,1]\,.
    \end{cases}
\end{equation}
  
\begin{teo}
\label{teo:2.1}
There exists a probability measure $p_t(x,\cdot)$, $t\geq 0$, with $x\in [0,1] $, on the Borel $\sigma$-algebra of $[0,1]$, such that the Cauchy problem (\ref{CauchyProblem}) has a unique solution of the form 
\[u(x,t)=\int_0^1 u_0(y)p_t(x,dy).\]
In addition, $p_t(x,\cdot)$ is the transition distribution of a strong Markov process, whose paths are right-continuous and have no discontinuities other than jumps. 
\end{teo}
\begin{proof}
    Due to the symmetry of $w$, $w(x,y)\leq w(y,x)$. The operator $\mathcal{W}$ then satisfies the positive maximum principle, i.e.\ if \( u \in C[0,1]\) and \( \max_{x \in [0,1]} u(x) = u(x_0) \geq 0 \), then \((\mathcal{W}u)(x_0) \leq 0 \), because 
\begin{align*}
\mathcal{W}u(x_0)& =
 \int_0^1w(x_0,y)(u(y)-u(x_0))dy
\\
&\leq \max_{y \in X} \left\{u(y) - u(x_0)\right\}\int_0^1 w(x_0,y)dy \leq 0 
\end{align*}
   
    Now, for any fixed \( \lambda > \|\mathcal{W}\| \), the operator 
    given as a von Neumann series
    \[ \frac{1}{1 - \frac{1}{\lambda} \mathcal{W}} = 1 + \frac{1}{\lambda} \mathcal{W} + \frac{1}{\lambda^2}\mathcal{W}^2 + \cdots + \frac{1}{\lambda^n}\mathcal{W}^n + \cdots \] 
    is linear and bounded. Consequently, 
    \[ 
    \text{R}\! \left( 1 - \frac{1}{\lambda} \mathcal{W} \right) = C[0,1]\,,
    \] 
    and thus \( \text{R}\! \left( 1 - \frac{1}{\lambda} \mathcal{W} \right) \) is dense in $C[0,1]$. Since all conditions for applying the Hille–Yosida–Ray theorem are fullfilled, \( \mathcal{W} \) generates a Feller semigroup \( \{ e^{t\mathcal{W}} \}_{t \geq 0} \).\\
    \\
    Now, we only have to use the fact that every Feller Semigroup has attached  uniformly stochastically continuous $C_0$-transition measures $p_t(x,dy)$ for all $t \in [0,\infty)$, for which  
     \[e^{t\mathcal{W}}u_0(x)=\int_0^1 u_0(y)p_t(x,dy) \]
     holds true.
    Then  use the connection between $C_0$-transition measures and Markov processes and we get 
    that the attached Markov process has paths that are right continuous and have no discontinuities other than jumps. Theses relations are outlined in \cite[Chapter III]{taira1991boundary}.
\end{proof}

\section{$p$-Adic diffusion on the the unit interval}

Here, a kernel function on the unit interval is defined which takes its arguments from the $p$-adic integers via the Monna map. It defines an integral operator which can be approximated via local averaging or by sampling, and this leads to discrete approximations of the solutions for the corresponding diffusion equations.

\subsection{$p$-adic Integral Operators via the Monna map}

The integral operator here is given by a kernel function
$k(x,y)$ defined via the Monna map $\rho$ as
\[
k(x,y)=f\left(\absolute{\rho^{-1}(x)-\rho^{-1}(y)}_p\right)\,,
\]
and the corresponding integral operator is given by
\[
D_fu(x)=\int_0^1k(x,y)(u(y)-u(x))\,dy
=\int_0^1f\!\left(\absolute{\rho^{-1}(x)-\rho^{-1}(y)}_p\right)(u(y)-u(x))\,dy
\]
for functions $u\colon [0,1]\to\mathbb{C}$.

{
\newcommand{\wlr}{\psi_{rnj}^{(p)}}
\newcommand{\wlp}{\widehat{\psi}_{rnj}^{(p)}}
\begin{teo}
\label{teo 4.1}
For $f\in L^1\left([0,1],\lambda\right)$,
    the wavelets $\wlr$ (\ref{genwavelet})
    supported in $[0,1]$ are eigenvectors of the operator 
$D_f$ on $L^2\left([0,1],\lambda\right)$.
    The corresponding eigenvalues are given by
    \[
\lambda_{r}=
(1-p^{-1})\sum\limits_{k=0}^{r-1}
f\left(p^{-k}\right)p^{-k}
+p^{-r}f(p^{-r})\,,
    \]
 where $r \in \mathbb{N}$.
\end{teo}
\begin{proof}

In order to see that the generalised eigenvalues $\wlr$ are eigenvalues, use the following commutative diagram:
\[
\xymatrix{
L^2\left([0,1],\lambda\right) \ar[r]^{D_f} \ar[d]_{\rho^*}
        & L^2\left([0,1],\lambda\right) \\
L^2\left(\mathbb{Z}_p,\mu\right) \ar[r]_{P_f}
        & L^2\left(\mathbb{Z}_p,\mu\right)\ar[u]_{{\rho^*}^{-1}}]
}
\]
in which the vertical arrows express the isometric isomorphism between the $L^2$-spaces revealed in Remark \ref{isometry}, and
which defines the $p$-adic operator $P$ on $L^2(\mathbb{Z}_p,\mu)$. It is readily seen to be given by
\begin{equation} \label{eq:padicoperator}
P_fh(\xi)=\int_{\mathbb{Z}_p}f\!\left(\absolute{\xi-\eta}_p\right)(h(\eta)-h(\xi))\,d\mu(\eta)\,,
\end{equation}
where $\mu$ is the Haar measure on $\mathbb{Z}_p$, normalised to
$\mu(\mathbb{Z}_p)=1$, and $h\colon\mathbb{Z}_p\to\mathbb{C}$ is a $p$-adic $L^2$-function. The corresponding wavelet eigenvalue can be seen as follows: First of all, the Kozyrev wavelets  supported in
a disk $D_r(a)$ of radius $p^{-r}$  centred in $a\in\mathbb{Z}_p$ are eigenfunctions of the operator $P_f$, as its kernel function depends only on the $p$-adic distance, and the corresponding eigenvalue is given by
    \begin{align*} \label{eigenvalue}
    \lambda_{r}& = 
    \int _{\left| a - \eta \right|_p >p^{-r}} 
    f\!\left(\absolute{a-\eta}_p\right)d\eta + 
    p^{-r} f\!\left(\absolute{a-(a + p^r)}_p\right)
    \end{align*}
according to \cite[Theorem 3]{Kozyrev2004}. The last summand simplifies to
$p^{-r}f(p^{-r})$, and the integral becomes a sum over integrals supported in circles centred in $a\in\mathbb{Z}$. This proves the asserted value for $\lambda_r$, and
this indeed proves the assertion, because the eigenvalues of $P_f$ and $D_f$ coincide due to the correspondence between
generalised wavelets and  Kozyrev wavelets   by \cite[Theorem 7]{Kozyrev2002}. Notice that this correspondence also holds true for $p\neq 2$, if Haar wavelets are replaced by (generalised) wavelets on $[0,1]$.   
\end{proof}
}

\begin{rem}
The isometric isometry between $L^2(\mathbb{R}_+,\lambda)$ and $L^2(\mathbb{Q}_p,\mu)$ is stated almost explicitly in \cite[Theorem 7.2]{Wilson2026} as an isometric isomorphism between $L^2(\mathbb{Z}_p,\mu)$ and $L^2([0,1],\lambda)$, but is in our viewpoint already covered by the results by Monna himself in \cite[Page 8]{Monna1952}, and possibly together with Kozyrev's wavelet-wavelet correspondence \cite[Theorem 7]{Kozyrev2002}. This correspondence becomes important in what now follows about the carrying over of a $p$-adic Markov process onto  the real interval $[0,1]$ through it.
\end{rem}

We now delve into the relationship between the (strong) Markov process generated by the operator \(D_f\) (Theorem \ref{teo:2.1}) and the Markov process attached to the operator \(P_f\) which is guaranteed to exists by Theorem \(2\) of \cite{Ledezma-Energy}. Nevertheless, the proof of the next Theorem does not rely on the existence of the process attached to \(P_f\), but rather, it appears naturally as the pull-back via the Monna map of the random process generated by \(D_f\). Moreover, the probability transitions of each process are related naturally by the Monna map as expected, as well as the respective semigroups.

\begin{teo}
    Let \(X_t\in [0,1]\) be the strong Markov process attached to the infinitesimal generator \(D_f\) with probability transition \(p_t(x,\cdot)\). Let \(P_f\) be the operator defined on \(L^2(\mathbb{Z}_p,\mu)\) defined by equation (\ref{eq:padicoperator}). Let \(X^{*}_t:=\rho^{-1}(X_t)\) be the pull-back of \(X_t\). 
    Then the following holds
    \begin{enumerate}
        \item The attached semigroups are unitarily equivalent:
        \begin{equation*}
            e^{tD_f}=\rho^{*-1} e^{tP_f}\rho^{*}
        \end{equation*}
        \item The pull-back \(X^{*}_t\in \mathbb{Z}_p\) is a strong Markov process on \(\mathbb{Z}_p\) with probability transition \(p_t^{*}(y,\cdot)\) satisfying
        \[p_t^{*}(x,B)=p_t(\rho(x),\rho(B)),\]for any Borel set \(B\subset \mathbb{Z}_p\).
    \end{enumerate}
\end{teo}
\begin{proof}
    Since the Monna map \(\rho:\mathbb{Q}_p\rightarrow\mathbb{R}_+\) is an injective measurable with inverse failing to exists in a set of measure zero, the generated \(\sigma\)-algebras \(\sigma(X_t)\) and \(\sigma(X^*_t)\) coincide, therefore the process \(X_t^{*}\) is a strong Markov process. Since \(D_f\) and \(P_f\) are bounded operators, and \(D_f^n=\rho^{*-1}P_f^n\rho^{*}\), the exponential formula \(e^{tD_f}=\sum_n \frac{t^n}{n!}D_f^n\), implies
    \[e^{tD_f}=\rho^{*-1} e^{tP_f}\rho^{*}.\]
    Consequently, for an arbitrary Borel set \(B\subset \mathbb{Z}_p\) the following holds true:
    \begin{equation*}
        \begin{split}
            p_t^*(x,B)&=\mathbb{P}(X_t^{*}\in B \,| \,X_s^*=x)\\
            &=\mathbb{P}(X_t\in \rho(B) \,| \,X_s=\rho(x))\\
            &=p_t(\rho(x),\rho(B))\,.
        \end{split}
    \end{equation*}
    This proves the assertions.
\end{proof}

\subsection{Finite approximations of $p$-adic diffusion on $[0,1]$}

The task here is to approximate solutions to the Cauchy problem (\ref{CauchyProblem}) defined on the interval $[0,1]$, by using
 related Cauchy problems on  discretisations of $[0,1]$. The relationship between the solutions of the discretised Cauchy problems and that of the original problem is established in Theorem \ref{teo:6.1} below. In order to discretise the interval $[0,1]$, we choose regular partitions
for each $n\in \mathbb{N}$,
given by
\begin{equation}\label{partition}
    A_n=\left\{\left[\frac{i}{n},\frac{i+1}{n}\right]\right\}_{i=0}^{n-1}.
\end{equation}
We specialize to regular refinements of partitions with growing $n$, that means if $m < n$, then $m \mid n$. 
For this, notice that for each pair $x,y\in [0,1]$,
there exists a unique sequence of sub-intervals
$I_n(x)\in A_n$ and $J_n(y)\in A_n$
such that 
\begin{equation*}
    \bigcap_{n=0}^{\infty}I_n(x)=\{x\}
    \hspace{0.4cm} \text{ and } \hspace{0.4cm} 
    \bigcap_{n=0}^{\infty}I_n(y)=\{y\}.
\end{equation*}
This allows us to identify each point in the unit interval
with a sequence of intervals $I_n$.
For a given $x\in[0,1]$,
the sequence of intervals $I(x)=(I_n(x))_{n\in\mathbb{N}}$
will be called the \emph{path} attached to $x$.

For the kernel function $w\in L^{\infty}([0,1]\times[0,1])$ that defines the diffusion operator, cf.\  (\ref{OurOperator})
\begin{equation*}
    C[0,1]\ni u(x)\mapsto Wu(x)=\int_{[0,1]} w(x,y)(u(y)-u(x))dy\,,
\end{equation*}
we now define transition rates $w_{I_n(x),J_n(y)}$ via the kernel function $w$.
There are two suitable ways to obtain such a discretisation of $w$. We develop the theory for both of these, using either averaging or a so-called sample sequence  $\mathcal{S}=\{x_I\}_{I\in A_n, n\in \mathbb{N}}$ which will be fixed throughout the paper.
\begin{defi}
    Let $A_n$ as in (\ref{partition}) be the $n$-regular partition of the unit interval $[0,1]$. For a given kernel function $w\in L^1([0,1]\times[0,1])$, define the average transition-rate matrix $W_n^a$ attached to the pair $(A_n,w) $ as
\[ 
w^{(n),a}_{I,J}:=[W_n^a]_{I,J}= \frac{1}{n^{2}}\int_{I\times J}w(x,y) dxdy,
\]
for $I,J\in A_n$.
The function defined by 
\[w_n^a(x,y):=\sum_{I,J\in A_n}w_{I,J}^{(n),a}\chi_{I}(x)\chi_J(y),\]
is called the \emph{$n$-th step average kernel} attached to the pair $(A_n,w)$. \newline

\end{defi}

\begin{prop}
    The entries of the average-transition-rate matrix $W_n^a$ attached to $(A_n,w)$ satisfy 
\begin{equation}
    \label{eqn:2}
    w(x,y)=\lim_{n\rightarrow \infty} w_{I_n(x),J_n(y)}^{(n),a}
\end{equation}
almost everywhere. 
\end{prop}

\begin{proof}
    For almost all $(x_0,y_0)\in [0,1]\times [0,1]$ the paths $I(x_0)$ and $J(y_0)$ satisfy $I_n(x_0)\rightarrow \{x_0\}$ and $J_n(y_0)\rightarrow\{y_0\}$. Therefore, using the Lebesgue Differentiation Theorem, it holds true that 
    \[  \lim_{n\rightarrow \infty} \frac{1}{n^2}\int_{I_n\times J_n}w(x,y) dxdy =w(x_0,y_0),
    \]
as asserted.
\end{proof}
We now introduce the continuous version of the matrix $W_n^a$ by  extending it  to an operator on the space $C[0,1]$ via the $n$-th step functions in the following way: Let $w_n^a$ be the $n$-th step average kernel. Then the assignment
\[
C[0,1]\ni u(x)\mapsto \mathcal{W}_n^au(x):=\int_{0}^1\left\{w_n^a(x,y)u(y)-w_n^a(y,x)u(x)\right\}dy
\]
is a well-defined linear bounded operator on $C[0,1]$. Notice, that again, this is a well-defined operator acting on the bigger space $L^{\infty}[0,1]$. 

\begin{rem}
Subsequently we write $w_n^a(x,y)$, $W_n^a$ and $\mathcal{W}_n^a$ all as $w_n(x,y)$, $W_n$ and $\mathcal{W}_n$, respectively, in order to relax the notation.
\end{rem}

Define
\[
X_n=\bigoplus\limits_{I\in A_n}\mathbb{C}\chi_I
\]
which is a finite-dimensional vector space.

\begin {defi}
\label{defi 6.2}
Let $X_{\infty}$ be the closure of $\mathcal{E}[0,1]$ in $L^{\infty}([0,1])$.  The projection operator will be defined in two ways. First,
\[
\mathbf{P}_n: X_{\infty}\rightarrow X_n
\]
is defined by
\[
X_{\infty}\ni \varphi\rightarrow \mathbf{P}_n(\varphi)(x)= \sum_{I\in A_n} \varphi_I^a\chi_{I}(x),
\]
where $\varphi_I^a$ is the averaged value of $\varphi(x)$ on $I \in A_n$. The embedding operator 
\[
\mathbf{E}_n:X_n\rightarrow X_{\infty}
\]
is defined as the identity map  $X_{\infty}\to X_\infty$ restricted to $X_n$.
\end{defi}

\begin{prop}
    Let $\mathfrak{M}_n$ be the Markov 
    process 
    attached to $W_n$. Then for given two states $I^{(n)},J^{(n)}\in A_n$  the transition probability of $\mathfrak{M}_n$ is given by the solution of the following Cauchy problem for $u(\cdot,t)\in X_n$, and $u(x,\cdot)\in C^1[0,\infty)$:
    \begin{equation}
\label{eqn:5}
    \begin{cases}
      \frac{\partial u}{\partial t} (x,t)=\mathcal{W}_{n}u(x,t)\\
      u(x,0)=\textbf{P}_nu_0(x)\in X_n
    \end{cases}\,.
\end{equation}
\end{prop}

\begin{proof}
    For this we only have to show that the matrix representation of $W_n$ in the finite-dimensional space $X_n$ 
    gives rise to the Master Equation (\ref{eqn:1}). This is clear by the following computation: Let \[
    u(x)=\frac{1}{n}\sum_{J\in A_n}f_J \chi_{J}(x)\in X_n\,,
    \]
    then 
\[
W_nu(x)=\sum_{J\in A_n}f_J\frac{1}{n}\int_{J}w_n(x,y)dy-\frac{1}{n}\sum_{J}f_J1_{J}(x)\sum_{K\in A_n}\int_{K}w_n(y,x)dy\,,
\]
whence for a given basis element $\frac{1}{n}\chi_{I}(x)$ for $I\in A_n$, we have 
\begin{equation*}
    \left\langle\frac{1}{n}1_{I}(x),W_nu(x)\right\rangle_{L^2([0,1])}= \sum_{J}\left\{f_Jw^{(n)}_{I,J}-f_Iw^{(n)}_{J,I}\right\}\,.
\end{equation*}
Hence (\ref{eqn:5}) is equivalent to (\ref{eqn:1}). Therefore, the transition probability attached to  (\ref{eqn:5}) is equal to the transition probability attached to the Markov chain $\mathfrak{M}_n$.   \end{proof}

Below, we will see that the solutions of the Cauchy problems attached to the matrices $W_n$ will converge uniformly to the solution of the one attached to $\mathcal{W}$. 
This result extends the results of \cite{PLC2021}. 
In order to prove this result, we introduce some preliminary definitions. Denote by $\mathcal{E}[0,1]$ the space of step functions equipped with the norm $||\cdot||_{\infty}$. It is a linear subspace of $L^{\infty}[0,1]$.

\label{teo:6.1}
\begin{teo}
Let $w\in L^{\infty}([0,1]\times[0,1])$ and
$w_n=w_n^a$ such that
$w_n\rightarrow w$ with respect to $|| \cdot ||_\infty$. If $u(x,t)$ is the solution of  
(\ref{CauchyProblem}) with initial condition $u_0(x)\in C[0,1]$,  and  $u_n(x,t)$ are  the solutions of (\ref{eqn:5}) with initial conditions $\textbf{P}_nu_0(x)$ for $n\in\mathbb{N}$, then the following  holds true:
\[\lim_{n\rightarrow \infty} \sup_{T \geq t\geq 0} ||u_n(x,t)-u(x,t)||_{\infty}=0. \]
\end{teo}

\begin{proof} 
The spatial variable is suppressed in the remainder, for notational convenience. By assumption $w_n$ converges to $w$ with respect to $|| \cdot ||_\infty$, whence
\[
\|\mathcal{W}_n - \mathcal{W}\|_{\mathrm{op}} \rightarrow 0 , \qquad n\to\infty,
\]
where $\|\cdot\|_{\mathrm{op}}$ denotes the usual operator norm. \\
We define the expression
\[
g(t) := u_n(t) - u(t)
   = T_n(t)\mathbf{P}_n u_0 - T(t)u_0,
\]
where $T_n(t)=e^{t\mathcal{W}_n}$ and $T(t)=e^{t\mathcal{W}}$ are the semigroups solving the Cauchy problems 
\eqref{CauchyProblem} and \eqref{eqn:5}.  
Define also
\[
\tilde{u}_0 := \mathbf{P}_n u_0.
\]
Then
\[
g(t) 
 = T_n(t)\tilde{u}_0 - T(t)\tilde{u}_0
   + T(t)\tilde{u}_0 - T(t)u_0
 = \tilde{g}(t) + T(t)(\tilde{u}_0 - u_0),
\]
where
\[
\tilde{g}(t)
 := (T_n(t) - T(t))\tilde{u}_0.
\]
Since $\tilde{u}_0 - u_0 =\mathbf{P}_n u_0 - u_0$ converges  in $C[0,T]$ with respect to $\|\cdot\|_\infty$ and $T(t)$ is bounded on $[0,T]$, the only contribution that remains to be estimated is $\tilde{g}(t)$.
We use semigroup theory to obtain an estimate for $\sup_{T\geq t \geq 0}||\tilde{g}(t)||_\infty $.
Differentiating,
\[
\tilde{g}'(t)
 = \mathcal{W}_n T_n(t)\tilde{u}_0 - \mathcal{W} T(t)\tilde{u}_0
 = \mathcal{W}_n \tilde{g}(t) + (\mathcal{W}_n - \mathcal{W})T(t)\tilde{u}_0.
\]
Since $T_n(0)=T(0)=I$, we have $\tilde{g}(0)=0$.  
Thus $\tilde{g}$ satisfies the linear differential equation
\[
\begin{cases}
\tilde{g}'(t) = \mathcal{W}_n\tilde{g}(t) + (\mathcal{W}_n - \mathcal{W})T(t)\tilde{u}_0,\\[2mm]
\tilde{g}(0) = 0.
\end{cases}
\]

By variation of parameters,
\begin{align*}
\tilde{g}(t)&
 = e^{t\mathcal{W}}\tilde{g}_0+ \int_{0}^te^{(t-s)\mathcal{W}_n}(\mathcal{W}_n-\mathcal{W})T(s)\tilde{u}_0\,ds
 \\
&=\int_0^tT_n(t-s)(\mathcal{W}_n-\mathcal{W})T(s)\tilde{u}_0\,ds.
\end{align*}
Hence,
\[
\sup_{0\le t \le T} \|\tilde{g}(t)\|_\infty
 \le \int_{0}^{T} 
   \|T_n(t-s)\|_{\mathrm{op}}
   \,\|\mathcal{W}_n - \mathcal{W}\|_{\mathrm{op}}
   \,\|T(s)\tilde{u}_0\|_\infty
\, ds.
\]
Both semigroups are uniformly bounded on $[0,T]$, and $\|\mathcal{W}_n-\mathcal{W}\|_{\mathrm{op}}\to 0$.
Combining the both above estimations,
\[
\sup_{0\le t\le T}\|u_n(t) - u(t)\|_\infty
 \le 
 \sup_{0\le t\le T}\|\tilde{g}(t)\|_\infty
 + \sup_{0\le t\le T}\|T(t)\|_{\mathrm{op}}\,
    \|\tilde{u}_0 - u_0\|_\infty.
\]
Therefore, $u_n$ converges uniformly to $u$ on $[0,T]$, completing the proof.
\end{proof}

\section{Vizualizing Diffusion processes 
on the $p$-adic integers}



The theory developed above allows us to implement the following
algorithm to numerically solve differential equations,
given by certain bounded integral operators.
These must be defined by specifying a radially symmetric
kernel and act on ultrametric spaces,
equivalent to the boundary of a $p$-adic tree.
This boundary serves as the domain of integration
and is approximated by the equivalence classes
given by a fixed depth $d$,
i.e. $\mathbb{Z}_p / p^d \mathbb{Z}_p$.
On a high level,
the algorithm is described as such:

\begin{verbatim}
INPUT: `p` a prime, `h` a natural number.
    The initial value 'f: dT -> X' and Operator D.

1) Decompose the function 'f' into the waveletbasis of 'T'.
2) Solve the eigenvalue problem associated to D.
3) Compute the evolution in time using the heat kernels.
4) Recombine the evolved components.
\end{verbatim}

Using the Monna isometry,
we can identify the leaves of a $p$-adic tree
with the intervals given by an iterative $p$-partition
of the unit interval up to as set of measure zero.
Thus the procedure calculates the time evolution
of a function with $p$-adic domain
under the given integral operator.
Isometrically equivalently,
this is approximately the time evolution
of a compactly supported function
with real domain.
For brevity denote $\mathbb{Z}_p / p^d
    := \mathbb{Z}_p / p^d \mathbb{Z}_p$.

\subsection{Breadthwise Decomposition}
Interestingly enough,
a full wavelet decomposition is too much
computational effort
-- as during Step {\tt 3)}\ in the algorithm above, all wavelets with
equally sized supports are scaled by the same factor,
since the operator $D$ is given by a radially symmetric kernel.
So it suffices to find components within the
\emph{breadthwise decomposition}
\[
L^2(\mathbb{Z}_p / p^d)
    = E_0^D \oplus E_{\lambda_1}^D \oplus \dots \oplus E_{\lambda_d}^D\,,
\]
where $0=:\lambda_{-1}, \lambda_0,\dots,\lambda_d$ are the eigenvalues of $D$.
\newline

The space $L^2(\mathbb{Z}_p/p^d)$ can be understood as
those locally constant functions on $\mathbb{Z}_p$
with radius of constancy not less than $p^{-d}$.
The embedding into $L^2(\mathbb{Z}_p)$ is given
by taking a function $f\in L^2(\mathbb{Z}_p/p^d)$
to the function $\overline f \in L^2(\mathbb{Z}_p)$ 
such that for all $n \in \mathbb{Z}_p / p^d$
\[
n\ni x\mapsto \overline f(x) := f(n)\in\mathbb{C}.
\]


%
Since $\mathbb{Z}_p$ is compact,
every locally constant function in $L^2(\mathbb{Z}_p)$ arises
from such an embedding,
given a sufficiently large $d \in \mathbb{N}$.
In particular this includes all Kozyrev-wavelets.
Notice that a wavelet with $\mathop{\text{supp}} \psi \subseteq n$
for $n \in \mathbb{Z}_p / p^d$ leads to vanishing
scalar products with $f \in L^2(\mathbb{Z}_p / p^d)$ as
\begin{equation}\label{eq:vanishing-scalar-product}
    \langle f, \psi\rangle
    = \int_{\mathbb{Z}_p} f(x) \psi^*(x) \mathop{d_p x}
    = \int_n f(x) \psi^*(x) \mathop{d_p x}
    = f(n) \int_n \psi^*(x) \mathop{d_p x} 
    = 0.
\end{equation}

Thus, the wavelet decomposition of a map
$f \in L^2(\mathbb{Z}_p / p^d)$
 only depends on the indicator $\Omega(|x|_p)$ and
those non-constant wavelets $\psi$ with support
containing any of the equivalence classes in $\mathbb{Z}_p/p^{d-1}$,
i.e.
\begin{equation}
    f(x)=\langle f,\Omega(|\cdot|_p) \rangle \Omega(|x|_p)
    + \sum_{j=0}^{d-1}
      \sum_{r=1}^{p-1}
      \langle f,\psi_{j,0,r} \rangle \psi_{j,0,r}(x).
\end{equation}

Since the eigenvalue for $\psi_{j,0,r}$
only depends on $j$ for integral operators $D$ on $\mathbb{Z}_p$ with
radially symmetric kernel,
we may define $f^{(-1)}(x) := \int_{\mathbb{Z}_p} f(z) \mathop{d_p z}$
and for $j \in \{0,\dots, d-1\}$ set
\begin{equation}
    f^{(j)} :=
    \sum_{r=1}^{p-1}
    \langle f, \psi_{j,0,r} \rangle \psi_{j,0,r}.
\end{equation}

\begin{lem}\label{lem:breadthwise-decomposition}
For $0 \leq \ell \leq d$ consider the function
$A^\ell f := f^{(-1)} + \sum_{j=0}^{\ell-1} f^{(j)}$. It satisfies
$$
    A^\ell f(n) = p^\ell \int_{n^{(\ell)}} f(x) \mathop{d_p x},
$$
where $n^{(\ell)} \in \mathbb{Z}_p / p^\ell$ is the unique equivalence
class for which $n\subseteq m$.
\end{lem}
\begin{proof}
Consider the integral $\int_{n^{(\ell)}} \psi_{j,a,r}(x) \mathop{d_p x}$
for $n \in \mathbb{Z}_p / p^d$.
If $\ell \leq j \leq d-1$,
then the $\mathop{\text{supp}} \psi_{j,a,r} = a + p^j \mathbb{Z}_p
    \subseteq n^{(\ell)}$,
so by the same argument as for
\cref{eq:vanishing-scalar-product}
the integral vanishes.
Otherwise $n^{(\ell)}
    \subset \mathop{\text{supp}} \psi_{j,a,r}$
and hence $\psi_{j,a,r} (x) = \psi_{j,a,r}(n)$
is constant on $n^{(\ell)}$ turning
the integral into $\int_{n^{(\ell)}} \psi_{j,a,r}(x) \mathop{d_p x}
    = p^{-\ell} \psi_{j,a,r} (n)$.
By definition $A^d f = f$, thus
\begin{align*}
    &\int_{n^{(\ell)}} f(x) \mathop{d_p x}
    = \int_{n^{(\ell)}} A^d f (x) \mathop{d_p x}
    = \int_{n^{(\ell)}} \left( A^\ell f(x)
        + \sum_{j=\ell}^{d-1} f^{(j)}(x)\right)
    \mathop{d_p x}\\
    =&\int_{n^{(\ell)}} A^\ell f(x) \mathop{d_p x}
        + \int_{n^{(\ell)}} \sum_{j=\ell}^{d-1} f^{(j)}(x) \mathop{d_p x}
    = \int_{n^{(\ell)}} A^\ell f(x) \mathop{d_p x}
    = p^{-\ell} A^\ell f(n).
\end{align*}
\end{proof}

This linear operator $A^\ell$ takes local averages
on $\mathbb{Z}_p / p^\ell$ to transform a map from
$L^2(\mathbb{Z}_p / p^d)$ to another from $L^2 (\mathbb{Z}_p / p^\ell)$.
Using the convention $A^{-1} := 0$,
one may compute the breadthwise decomposition $\{f_\ell\}_{\ell=0}^{d}$
for $f$ as such
$f_\ell = A^{\ell} f - A^{\ell-1} f$.

\subsection{Detailed Description of the Algorithm}
The previous Lemma \ref{lem:breadthwise-decomposition}
informs the following algorithm
to obtain the breadthwise-decomposition.
Via the Monna map,
one may impose a linear order on $\mathbb{Z}_p / p^d$
indicated by the essential infima,
giving a formal interpretation to the tuple
$f_0=(f_0^1, \dots, f_0^N)$ with $N := p^d$
as representing a function within $L^2(\mathbb{Z}_p /p^d)$.
To form $A^{d-1} f$,
partition $f_0$ into $p$-sized disjoint blocks
and compute the average of each block.
Iterate this to compute $(A^df_0,..., A^0f_0)$.

The same procedure can be generalized to $Q$-adic trees,
where the branching number is not constant for every node.
Instead a list of primes $Q=(q_\ell)_{\ell=1}^d$ is given
and the branching number is constant for every level $\ell$.
In such a case $N = \prod_{\ell=1}^d q_\ell$ must be assumed.
This is imlemented in python below.
By taking $(1-p^{-1})$ into the sum from Theorem \ref{teo 4.1}
one can already hint at the eigenvalue for $Q$-adic domains.

\begin{minipage}{18cm}\tt
from math import prod
\\
def decompose(Q: list, f: list) -> list[list]:
\\
\hspace*{1cm}    N=len(f); fs=[f,]
\\
\hspace*{1cm}    assert prod(Q) == N, f"{Q=} must factorize {N=} without residue!"
\\
\hspace*{1cm}    for q in Q[::-1]:
\\
\hspace*{2cm}        cur = fs[-1]; n=len(cur)
\\
\hspace*{2cm}        fs.append([ prod(cur[i:i+q]) for i in range(0,n,q) ])
\\
\hspace*{1cm}    Af=fs[::-1]; out=[Af[0]]  {\# fixes indexing: Af[i][j] = A\^{}i f(j)}
\\
\hspace*{1cm}    for i,q in enumerate(Q[1:]):
\\
\hspace*{2cm}        n = len(Af[i])
\\
\hspace*{2cm}        out.append([ Af[i][j] - Af[i-1][j//q] for j in range(n) ])
\\
\hspace*{1cm}    return out
\end{minipage}

The space needed for this is given by
$\sum_{k=0}^{d} p^k = \frac{p^{d+1}-1}{p-1}$
and thus within $\mathcal{O}(N)$.
For time complexity consider that
each iteration can be performed in exactly $N$ lookups,
$\frac{N}{p}$ writes,
$\frac{N(p-1)}{p}$ additions and $\frac{N}{p}$ divisions,
giving a complexity of $\mathcal{O}(N)$,
or to be more precise:
\[
    \underbrace{\sum_{k=1}^d \frac{p^k (p+1)}{p}}_{\text{I/O-ops}}
        + \underbrace{\sum_{k=1}^d p^k}_{\text{flops}}
    = \frac{(2p+1)(p^d-1)}{p-1}
\]

After the breadthwise decomposition is obtained
and the eigenvalues are computed as shown in Theorem \ref{teo 4.1},
the evolution in time may be computed by the heat kernels using
\begin{equation}\label{eq:time-evolution}
    f(t,x) = \sum_{r=-1}^d e^{- \lambda_r t} f_0^{(r)}(x).
\end{equation}

\subsection{Two Example Operators}
The algorithm explained above is demonstrated by
evolving a common initial value
$f_0$
via two families of kernels.
It is applied to the real function
$f_0(x) = e^{-16 \cdot (x-\frac{1}{2})^2}$
by discretizing it to a regular partition
of the unit interval.
For the examples $p=3$ is chosen and the $3$-adic domain
is approximated by a tree of height $h=6$.
The computational overhead lies in computing
this decomposition and the eigenvalues to the operator.
After these have been computed,
a computer search approximates the equilibrium time
given by $T := \inf \{t > 0 : |f_t|_\infty < 10^{-3}\}$.

The simulations presented below
take $N = p^h = 729$ time steps to generate an image,
but since linear subdivisions of $[0,T]$ rarely produce
interpretable images,
a parameter $\tau$ is introduced
which determines the length of a timestep
by normalizing the geometric progression
$(1+\tau)^n$ for $n=0,\dots, N-1$
to lie within the desired interval.
The "correct" value for $\tau$ lies in the eye of the beholder,
thus it is determined via trial and error.

\begin{figure}[h]
    \centering
    \begin{subfigure}{0.32\textwidth}
        \includegraphics[width=\textwidth]{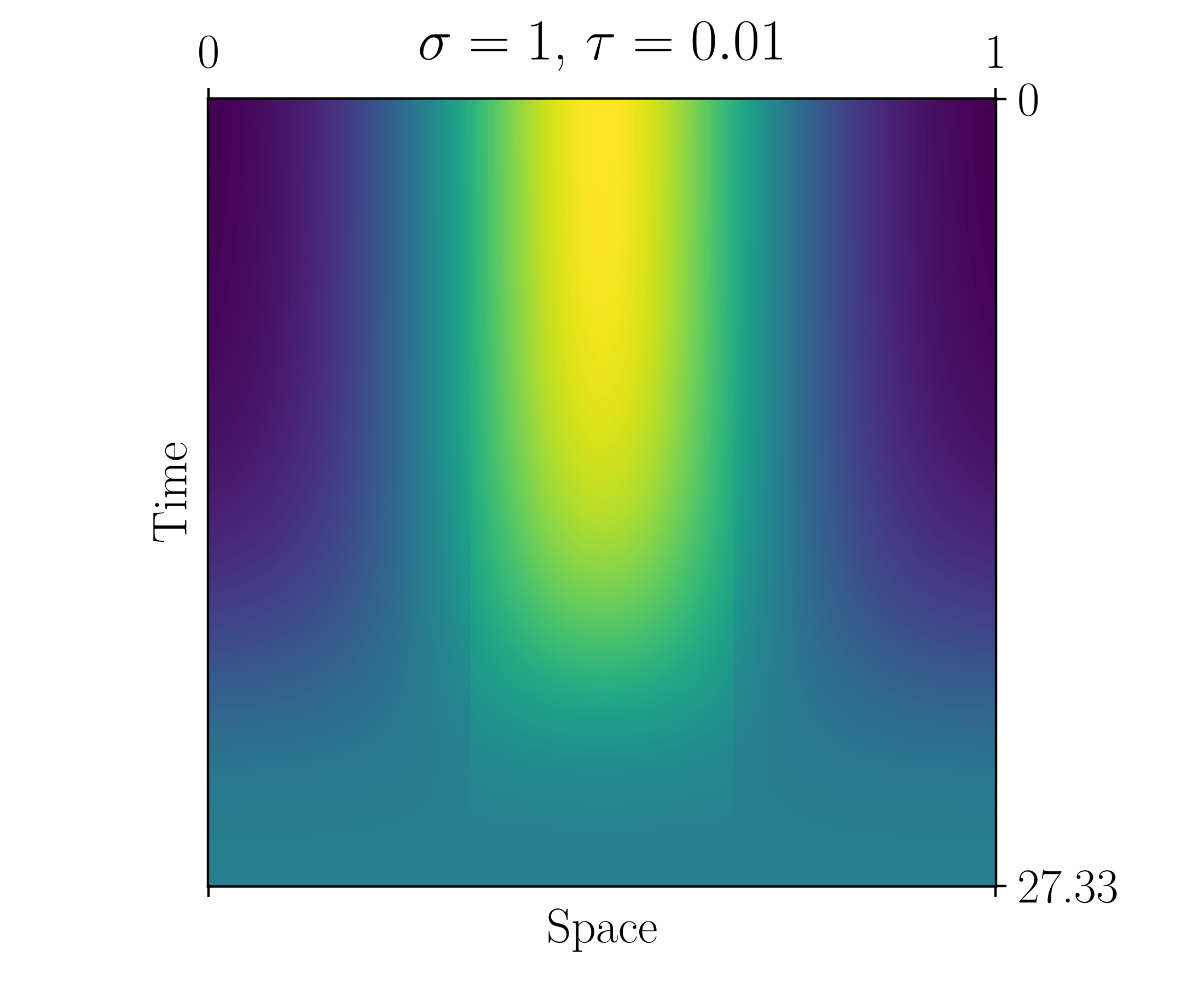}
    \end{subfigure}
    \hfill
    \begin{subfigure}{0.32\textwidth}
        \includegraphics[width=\textwidth]{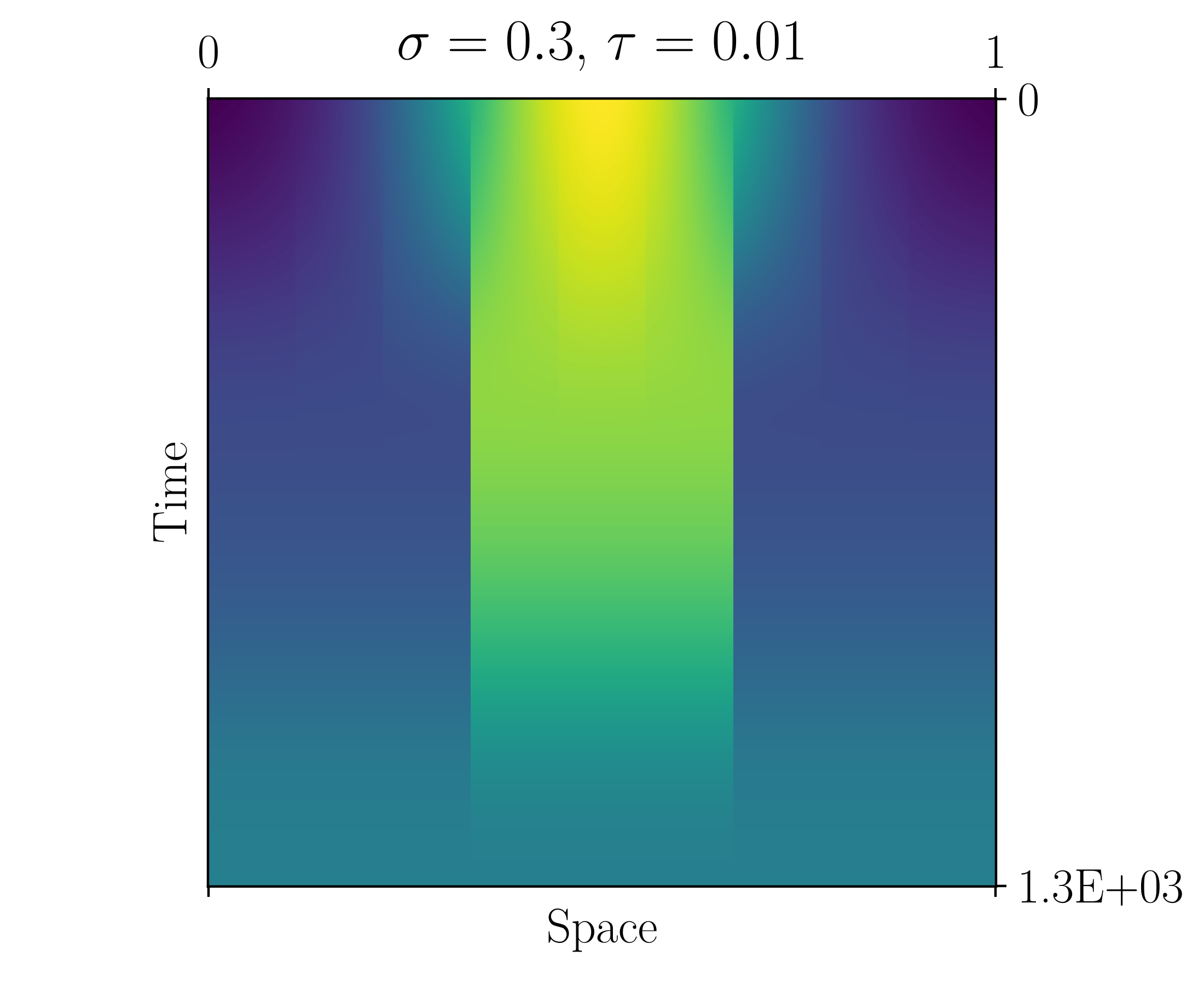}
    \end{subfigure}
    \hfill
    \begin{subfigure}{0.32\textwidth}
        \includegraphics[width=\textwidth]{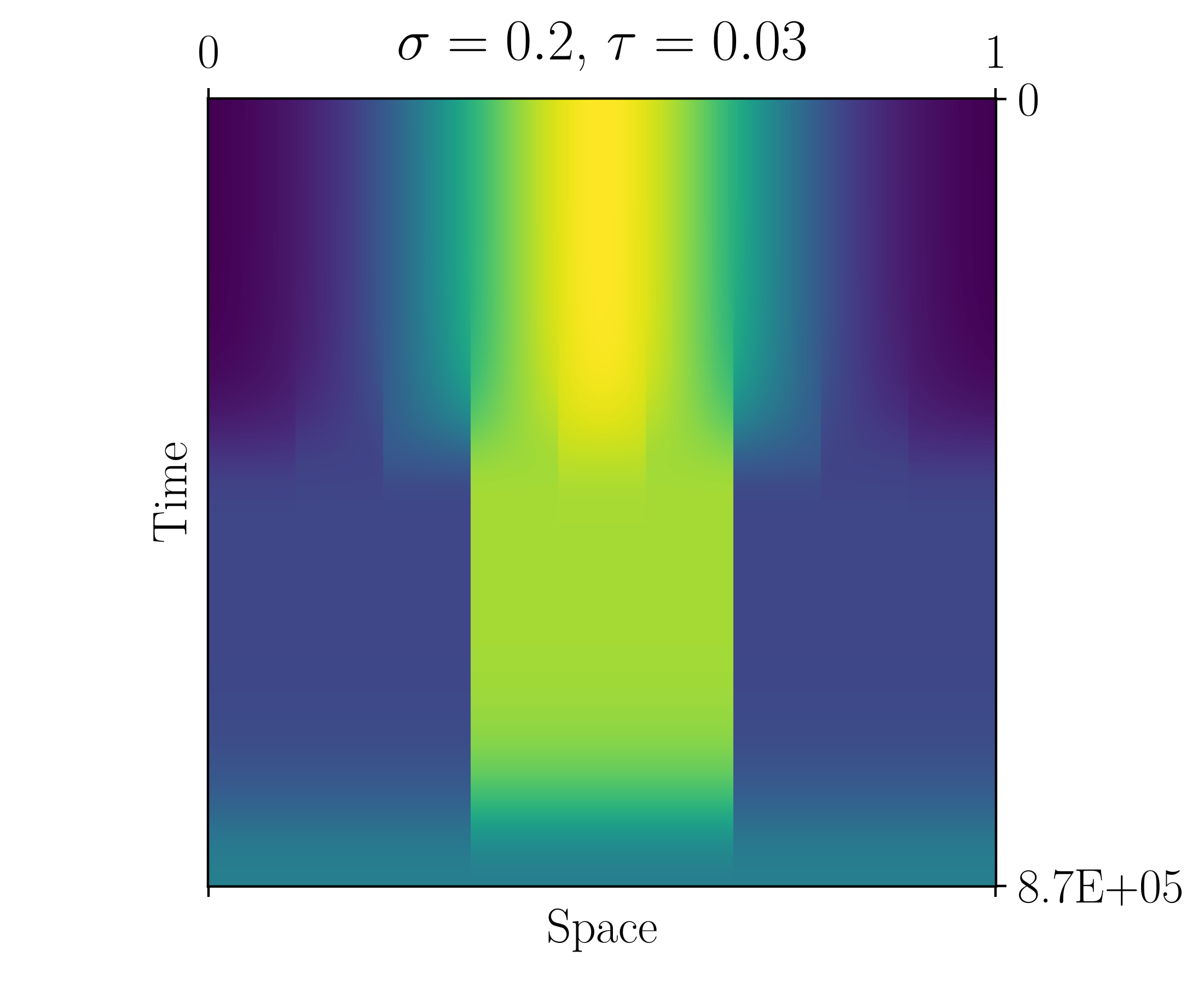}
    \end{subfigure}
    \caption{Evolution of $f_0$ by the kernel
        $k_\sigma(|x-y|_p)= \frac{1}{\sqrt{2 \pi \sigma^2}}
            e^\frac{-|x-y|_p^2}{2 \sigma^2}$
    }
    \label{fig:gaussian-kernel}
\end{figure}

\begin{figure}[h]
    \centering
    \begin{subfigure}{0.32\textwidth}
        \includegraphics[width=\textwidth]{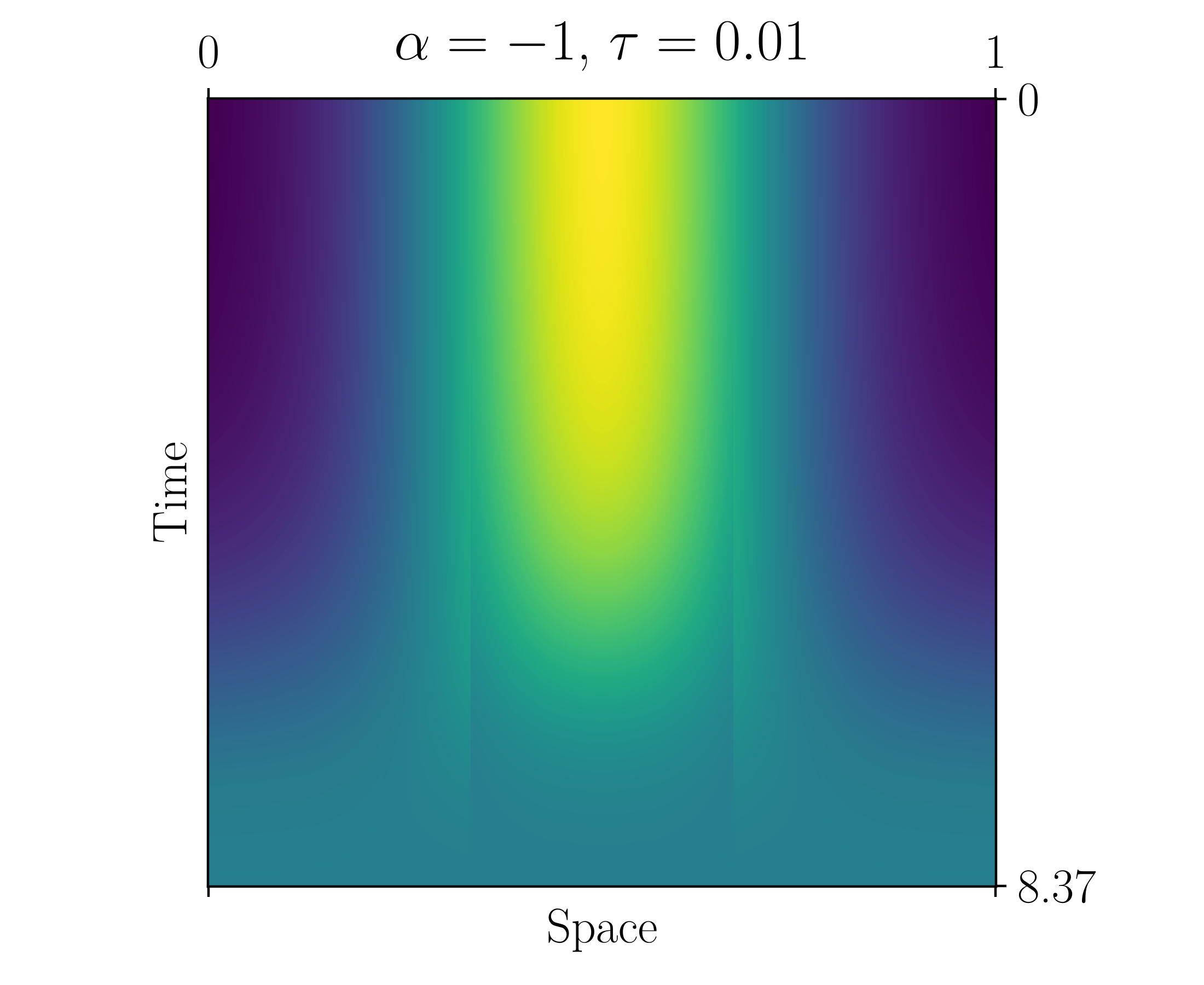}
    \end{subfigure}
    \hfill
    \begin{subfigure}{0.32\textwidth}
        \includegraphics[width=\textwidth]{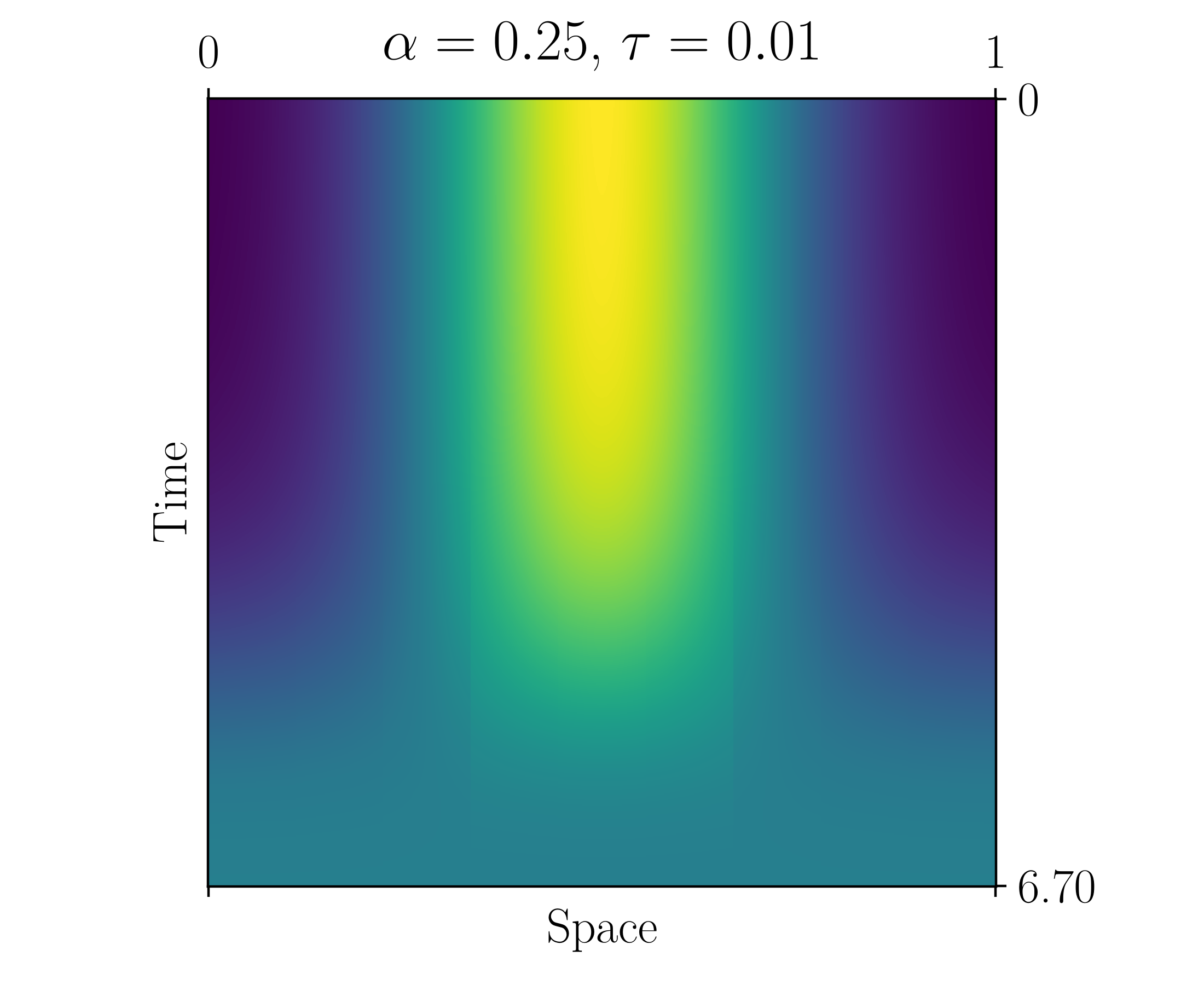}
    \end{subfigure}
    \hfill
    \begin{subfigure}{0.32\textwidth}
        \includegraphics[width=\textwidth]{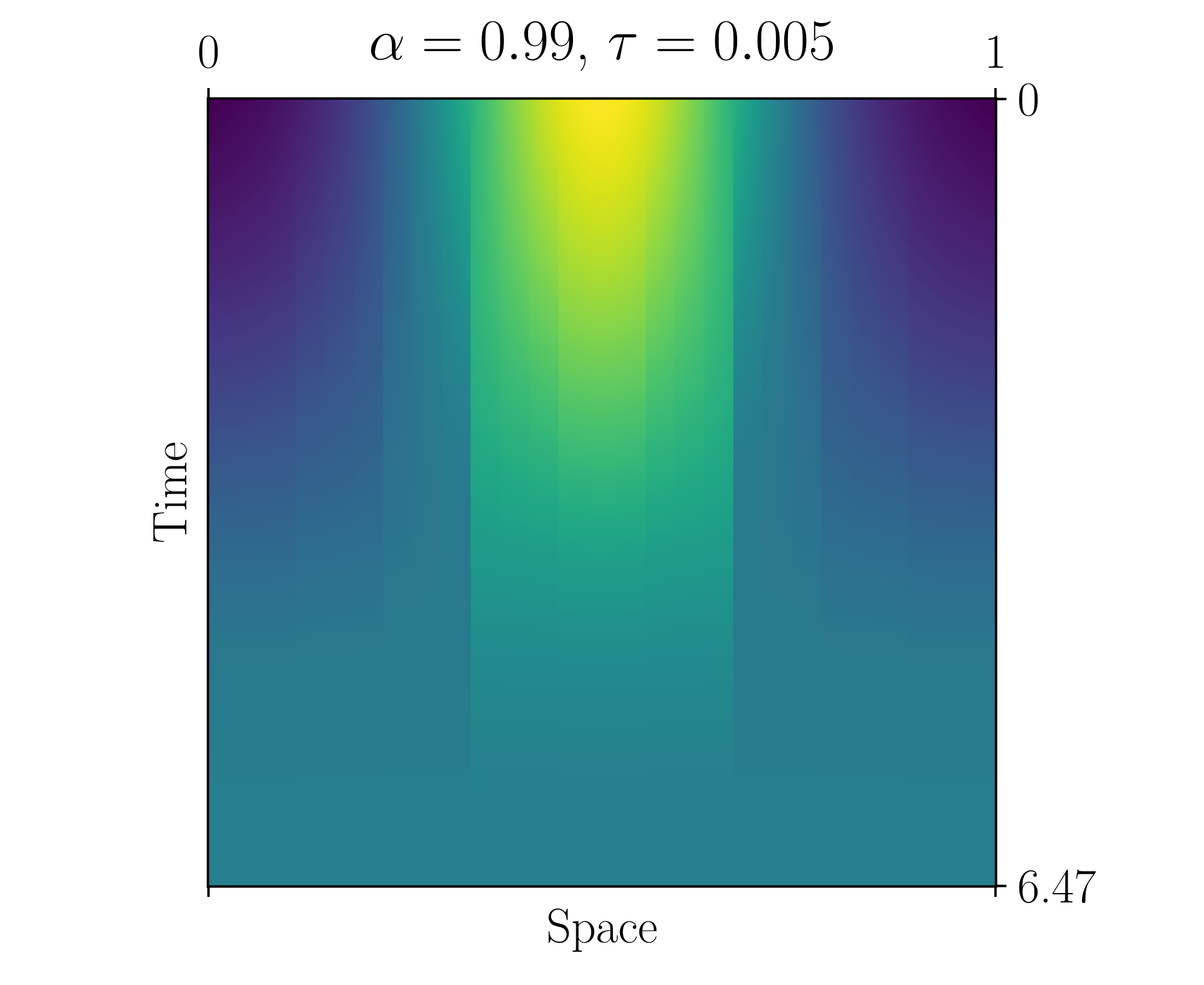}
    \end{subfigure}
    \caption{Evolution of $f_0$ by the kernel
        $k_\alpha(|x-y|_p)= |x-y|_p^{-\alpha}$
    }
    \label{fig:power-kernel}
\end{figure}

Figure \ref{fig:gaussian-kernel} depicts the time evolution
under operators defined by radial functions taking values
of a normal distribution with standard deviation $\sigma > 0$
at points of the form $p^{-k}$,
yielding a bounded operator.
In Figure $\ref{fig:power-kernel}$ the operators are given by
an inverse power law which gives a bounded operator for $\alpha < 1$. 
For $\alpha=0$ (and large values of $\sigma$)
the kernel is (almost) constant and thus models
an exponential decay of the mean-zero component.

All operators, except the one having $k_{-1}$ for its defining radial function,
exhibit eigenvalues $\lambda_r$ which grow
monotonically in $r$.
Such operators produce evolutions in time,
where perturbations on small scales of the initial value
vanish faster than those on larger scales.
As can be seen in the graphic above,
where as time $t$ progresses 
$f_t$ becomes increasingly well approximated
by fewer terms of Equation (\ref{eq:time-evolution}),
turning $f_0$ (almost) locally constant
on succeedingly larger intervals.
This discretization effect becomes more drastic,
as $\sigma$ is decreased
(favoring local over global interactions)
and thus dilating the equilibrium time for $f_0$,
as it decomposes with large contributions from
the first couple of eigen spaces.

\section*{Acknowledgements}
The authors want to thank 
Martin Breunig, 
Bastian Erdn\"u{\ss}, Markus Jahn, 
David Weisbart, 
and
Wilson Z\'{u}\~{n}iga-Galindo
for fruitful discussions.
This research is partially funded by the Deutsche For\-schungsgemeinschaft under 
project number 469999674.

\bibliographystyle{plain}
\bibliography{biblio}

\end{document}